\documentclass[reqno,a4paper]{amsart}
\usepackage[active]{srcltx}
\usepackage{fancyhdr}
\usepackage{times}
\usepackage{mathrsfs} 
\usepackage{latexsym,amsopn,amssymb,amsmath}
\usepackage{amsthm}
\usepackage{amsfonts,amsbsy,amscd,stmaryrd}
\usepackage{paralist}
\usepackage{hyperref}
\hypersetup{colorlinks=true,linkcolor=Emerald,urlcolor=blue} 
\usepackage[usenames,dvipsnames,svgnames,table]{xcolor}
\newcommand{\C}{\mathbb{C}}

\newcommand{\N}{\mathbb{N}}

\newcommand{\R}{\mathbb{R}}

\newcommand{\ca}{\mathcal{A}}

\newcommand{\ce}{\mathcal{E}}

\newcommand{\ch}{\mathcal{H}}

\newcommand{\cs}{\mathcal{S}}

\def\eith{\e^{\i tH}}

\def\d{\mathrm{d}}

\def\e{\mathrm{e}}

\def\i{\mathrm{i}}

\def\braket#1#2{\langle{#1}|{#2}\rangle}

\def\jap#1{\langle {#1} \rangle}

\DeclareMathOperator*{\slim}{s-lim}		

\def\sign{\text{ sign\,}}

\def\qed{\hfill $\Box$\medskip}
\long\def\symbolfootnote[#1]#2{\begingroup%
\def\thefootnote{\fnsymbol{footnote}}\footnote[#1]{#2}\endgroup}
\newtheorem{theorem}{Theorem}

\newtheorem{proposition}[theorem]{Proposition}
\newtheorem{corollary}[theorem]{Corollary}

\newtheorem{definition}[theorem]{Definition}
\newtheorem{remark}[theorem]{Remark}

\newtheorem*{theorem*}{Theorem}  
\makeatletter
\newcommand{\labitem}[2]{%
\def\@itemlabel{\textit{#1}.}
\item
\def\@currentlabel{#1}\label{#2}}
\makeatother

\makeatletter
\newcommand{\labitemi}[2]{%
\def\@itemlabel{\textit{#1}}
\item
\def\@currentlabel{#1}\label{#2}}
\makeatother

\begin{document}

\title[Abstract theory of decay estimates: perturbed Hamiltonians]{Abstract theory of decay estimates: perturbed Hamiltonians}
\author[M. Larenas]{Manuel Larenas}
\address{Rutgers University, Department of Mathematics, 110 Freylinghuysen Road, Piscataway, NJ 08854, U.S.A.}
\email{mlarenas@math.rutgers.edu}
\thanks{The first author was partially supported by NSF DMS-1201394.}
\author[A. Soffer]{Avy Soffer}
\email{soffer@math.rutgers.edu}
\date{}
\begin{abstract} For two self-adjoint operators $H,A$ we show that a general commutation relation of type $[H,\i A]=Q(H)+K$, in addition to regularity of $H$ and Kato-smoothness of $K$, guarantee {\em pointwise} in time decay rates of diverse order. The methodology is based on the construction of a modified conjugate operator $\tilde{A}$ that reduces the problem to previously developed estimates when $K=0$. Our results apply to energy thresholds and do not rely on resolvent estimates. We discuss applications for the Schr\"odinger equation (SE) with potential of critical decay, and for the free SE on an asymptotically flat manifold.
\end{abstract}

\maketitle

\tableofcontents

\section{Introduction}

In the spectral analysis of self-adjoint operators, methods relying on the positivity of a commutator have very important applications. This approach can be traced back to the work of Putnam in 1967 \cite{P}, whose main result relates the condition $[H,\i A]\geq 0$ with the absolute continuity of the range of $[H,\i A]$. The number of applications of this setting is greatly restricted by the boundedness of the conjugate operator $A$ and the global assumption on the commutator. In the fundamental work of Mourre in 1981 \cite{Mo}, the assumptions are more flexible and admit a variety of extensions. In his work, it is required that $[H,\i A]$ is dominated by $H$ (for $A$ only self-ajoint) and the positivity assumption is represented by the so-called \emph{strict Mourre estimate} $E(J)[H,\i A]E(J)\geq a E(J)$, where $E$ is the spectral measure of $H$, $a$ is a positive constant and $J$ is a Borel set of $\R$. The main consequences of these conditions are a \emph{limiting absorption principle}, that is, a control of the resolvent of $H$ close to the real axis, and the absolute continuity of the spectrum of $H$ in the interval $J$. This framework also admits a compact perturbation on the right hand side, which results in the possible presence of discrete spectrum in $J$.  It was later shown that these commutator estimates also imply time decay estimates, and optimal propagation estimates \cite{SS,HSS}.

The \emph{weakly conjugate operator method} developed in \cite{BGM,BKM} can be seen as an extension of Putnam's work with elements of Mourre theory. Here the estimate $[H,\i A]>0$, along with an assumption of regularity of type $[H,\i A]\in C^1(A)$, lead to a limiting absorption principle. An important extension was proposed by Richard \cite{Ri} which deals with operators that are not purely absolutely continuous. In this case the positive commutator above is replaced by the condition $[H,\i A]>cH$ for $c\geq 0$. We also mention the work in \cite{MRT,MT} where the main assumption is of type $[H,\i A]\geq 0$, in other words, the injectivity may fail.

In this paper and its predecessor \cite{GLS}, we develop a methodology that uses a commutation relation (not necessarily involving positivity) to derive  pointwise decay estimates in time for an abstract Hamiltonian $H$. Estimates of this type were first obtained by Jensen and Kato \cite{JK} for $H=-\Delta+V$ in dimension $n=3$. They proved estimates in weighted $L^2(\R^3)$ spaces through a resolvent expansion around zero. A unified approach for all dimensions was later developed by Jensen and Nenciu \cite{JN1}. Journ\'e, Soffer and Sogge \cite{JSS} proved global estimates for $H=-\Delta+V$ in $n\geq 3$. Numerous extensions followed, see e.g. \cite{DSS,Ya,EGG,Gol,GS,Sch} and cited references.

The conjugate operator method plays a crucial work in the derivation of decay estimates and resolvent estimates in many different situations (see for instance \cite{SS,HSS,Ger}). However, these techniques are less useful for problems on manifolds. In fact, most of the results in this context are of local decay type and Strichartz estimates \cite{RT,BSo1,BSo2,DR,Ta}. In this work we show that such results can be used to derive Kato-smoothness of the terms in the commutation relation and subsequently use this property to obtain pointwise decay estimates in an abstract setting.

The pointwise decay estimates of the Kato-Jensen type play an important role in applications. The $L^p$ decay estimates are proven using them, as well as extensions to time dependent hamiltonians of the charge transfer type. It was remarked by Ginibre, that the Kato-Jensen estimates imply a weak version of the $L^p$ estimates, in certain cases.
However, the techniques used to prove such estimates relies on detailed and explicit knowledge of the Green's function of the unperturbed hamiltonian.
The abstract approach may therefore allow extensions to more general classes of hamiltonian operators, which can not be simply represented as a perturbation of a solvable system.

{\bf Acknowledgment:} We would like to thank V. Georgescu for helpful remarks.

\section{Preliminaries}

Let $A$ and  $H$ be two self-adjoint operators on a Hilbert space $\ch$. We are interested in extending the results of \cite{GLS} to the case where $[H,\i A]$ is not necessarily equal to a function of $H$. In this work we will assume a more general commutation relation which will yield similar decay estimates by means of a suitable adaptation of the conjugate operator $A$.

We now review some standard definitions in functional analysis. As usual, we write $\jap{x}=(1+x^2)^{1/2}$. Denote $\ch^1=D(H)$ the domain of $H$ and consider its adjoint space $\ch^{-1}=D(H)^*$. The resolvent of $H$ is defined as $R(z)=(H-zI)^{-1}$ for $z\in\rho(z)$ the resolvent set of $H$. For $P$ a bounded operator, we shall say that $P$ \emph{commutes with} $H$ if for any $t\in\R$ the relation $P\e^{\i tH}=\e^{\i tH}P$ holds in $B(\ch)$. This is equivalent to $P\varphi(H)=\varphi(H)P$ for any bounded Borel function $\varphi:\R\to\C$. If $Q$ is unbounded and $D(Q)\supset\ch^1$ then we say that $Q$ commutes with $H$ if the above identity holds in $B(\ch^1,\ch)$ for $Q$.

Consider $Q$ a densely defined operator on $\ch$ with $D(Q)\supset\ch^1$. We say that $Q$ is $H$\emph{-bounded with relative norm }$a$ if for some $a,b\in\R$ one has $\|Q\psi\|\leq a\|H\psi\|+b\|\psi\|$, for all $\psi\in\ch^1$.

If $S$ is a bounded operator on $\ch$ then we denote $[A,S]_\circ$ the
sesquilinear form on $D(A)$ defined by
$[A,S]_\circ(u,v)=\braket{Au}{Sv}-\braket{u}{SAv}$. We
say that \emph{$S$ is of class $C^1(A)$}, and we write $S\in C^1(A)$,
if $[A,S]_\circ$ is continuous for the topology induced by $\ch$ on
$D(A)$ and then we denote $[A,S]$ the unique bounded operator on $\ch$
such that $\braket{u}{[A,S]v}=\braket{Au}{Sv}-\braket{u}{SAv}$ for all
$u,v\in D(A)$. We consider now the rather subtle case of unbounded operators. Note
that we always equip the domain of an operator with its graph
topology.  If $H$ is a self-adjoint operator on $\ch$ then
$[A,H]_\circ$ is the sesquilinear form on $D(A)\cap D(H)$ defined by
$[A,H]_\circ(u,v)=\braket{Au}{Hv}-\braket{Hu}{Av}$. A convenient definition of the $C^1(A)$ class for any self-adjoint
operator is as follows. Let $ R(z)= (H-z)^{-1}$ for $z$ in the
resolvent set $\rho(H)$ of $H$.  We say that \emph{$H$ is of class
  $C^1(A)$} if $R(z)\in C^1(A)$ for some (hence for all)
$z\in\rho(H)$. In this case, Proposition 6.2.10 in \cite{ABG} shows that $D(A)\cap\ch^1$ is dense in $\ch^1$ and hence $[A,H]_\circ$ extends to a uniquely determined continuous sesquilinear form $[A,H]$ on $\ch^1$. For further properties and examples of the $C^1$ regularity we refer to \cite{ABG} and \cite{GG}.

Finally, we discuss the notion of Kato-smoothness. We shall say that a closed operator $E$ is $H$\emph{-smooth on the range of a bounded operator} $P$ if and only if for each $\psi\in\ch$ and each $\epsilon\neq 0$, $R(\lambda+\i\epsilon)P\psi\in D(E)$ for almost all $\lambda\in\R$ and moreover $$\|E\|_H^2=\sup_{\|\psi\|=1}\frac{1}{4\pi^2}\int_{-\infty}^{\infty}\left(\|ER(\lambda+\i\epsilon)P\psi\|^2+\|ER(\lambda-\i\epsilon)P\psi\|^2\right)\d\lambda<\infty.$$
In particular $E$ is $H$-smooth on the range of $P$ if and only if for all $\psi\in\ch$, $\e^{\i tH}P\psi\in D(E)$ for almost every $t\in\R$ and $$\int_{-\infty}^{\infty}\left\|E\e^{-\i tH}P\psi\right\|^2\d t\leq 2\pi\|E\|_H^2\|P\psi\|^2.$$

In several applications the smoothness can be controlled more precisely using Sobolev norms as follows. Let $M$ be an $n$-dimensional Riemannian manifold with a smooth Riemannian metric $g_{ij}$. Consider the Hilbert space $\ch:=L^2(M)$ with the inner product defined as $\braket{\phi}{\psi}=\int_M \phi(x)\overline{\psi(x)}\d g$, where $\d g:=\sqrt{\textrm{det }g_{ij}}$. Denote by $\|\cdot\|$ the norm induced by this inner product, that is, $\|\phi\|=\int_M |\phi(x)|^2\d g$. The self-adjoint operator $H$ defined on $\ch$ enjoys the standard functional calculus and one can define the homogeneous Sobolev norms $\|\psi\|_{\dot{H}^s(M)}:=\||H|^{s/2}\psi\|$, for $0\leq s\leq 1$.

We then generalize the notion of $H$-smoothness taking the supremum on a subspace of $\ch$ instead of the whole space. We will say that a closed operator $E$ is $|H|^s$\emph{-smooth on the range of a bounded operator} $P$ if and only if for all $\psi\in\dot{H}^s(M)$, $\e^{\i tH}P\psi\in D(E)$ for almost every $t\in\R$ and $$\int_{-\infty}^{\infty}\left\|E\e^{-\i tH}P\varphi\right\|^2\d t\leq C_E\|P\psi\|_{\dot{H}^s(M)}^2.$$

\section{Assumptions}

Let $H$ and $A$ be two self-adjoint operators on the Hilbert space $\ch=L^2(M)$, where $M$ is a smooth $n$-dimensional Riemannian manifold equipped with the standard $L^2$-inner product $\braket{\cdot}{\cdot}$ and norm $\|\cdot\|$. Let $P$ be an orthogonal projection and $Q$ be an $H$-bounded operator with relative norm $a\geq 0$. $P$ and $Q$ commute with $H$. Let $E,F$ be linear operators such that $D(E)\cap D(F)\supset P\ch^1$. The main assumptions of this paper are established as follows.

\begin{itemize}
\labitemi{(H)}{H} \begin{itemize}
\item $H$ is of class $C^1(A)$
\item $H$ and $A$ satisfy the commutation relation $P[H,\i A]P=P(Q+K)P$ for $K\equiv F^*E$, in the sense that for all $\phi,\,\psi\in\ch^1$
\begin{itemize}
\item[]$(\phi,P[H,\i A]P\psi)=(P\phi,QP\psi)+(FP\phi,EP\psi).$
\end{itemize}
\item $K$ is symmetric on $\ch^1$, that is, $\braket{\phi}{K\psi}=\braket{K\phi}{\psi}$
\item $K\ch^1\subset\ch$ and $\jap{H}^{-s/2}K\jap{H}^{-s/2}$ is a bounded operator on $\ch$, for some $s>0$
\item $E$ and $F$ are $|H|^s$-smooth on the range of $P$
\end{itemize}
\end{itemize}

\begin{remark}
Since $H$ is of class $C^1(A)$, the sesquilinear form $[H,\i A]_\circ$ on $\ch^1\cap D(A)$ extends to a continuous operator in $B(\ch^1,\ch^{-1})$. Then the commutation relation of \ref{H} implies that the sesquilinear form $(P\phi,QP\psi)+(FP\phi,EP\psi)$ restricted to $\ch^1\cap D(A)$ also extends to a bounded operator in the same space. Therefore, the commutation relation can be written at the level of operators in $B(\ch^1,\ch^{-1})$ as $P[H,\i A]P=P(Q+K)P$, for a densely defined symmetric operator $K$.
\end{remark}

For some of the decay estimates we will require an additional smoothness condition.

\begin{itemize}
\labitemi{(Ha)}{Ha} The sesquilinear form $P[A,K]_\circ P$ defined on $D(A)\cap D(K)$ satisfies the identity $P[A,K]_\circ P(\phi,\psi)=(F'P\phi,E'P\psi)$, for all $\phi,\psi\in D(A)\cap D(K)$. Here $E',F'$ hold the same properties of $E,F$ in assumptions \ref{H}. Moreover, $P[A,K]_\circ P$ extends to a densely defined operator $K'$ such that $\jap{H}^{-s/2}K'\jap{H}^{-s/2}$ is bounded on $\ch$.
\end{itemize}


The commutation relation of \ref{H} is the crucial assumption of this paper. It replaces the identity of type $[H,\i A]=\varphi(H)$ in \cite{GLS}, with a much more general expression. Indeed, it now admits an operator $Q$ (controlled by $H$) in addition to a Kato-smooth perturbation. Observe that the $C^1(A)$ regularity provides a suitable framework in this case as well.

The strategy to derive decay estimates under these generalized conditions will be based on the construction of a new conjugate operator $\tilde{A}$ which will simplify the commutation identity, thus reducing the problem to our preceding estimates. The remaining assumptions of \ref{H} will justify this construction and other algebraic manipulations.

\section{The conjugate operator}

This section aims to construct the conjugate operator $\tilde{A}$. Using the functional calculus define for $s>0$ the cut-off $h_s(H):=\jap{H}^{-s/2}$. For any operator X denote $X_h:=h_s(H)Xh_s(H)$ (omitting the parameter $s$ in $X_h$ for brevity).

We first consider the operator $A_h$. Note that since $H\in C^1(A)$, the operator $h_s(H)$ is of class $C^1(A)$ for $s$ large enough (see the comment before Theorem 3.7 in \cite{GLS}). Moreover, by Lemma 6.2.9 in \cite{ABG}, the $C^1(A)$ condition for $h_s(H)$ holds for any $s>0$. This follows from the continuity of the form $[A,h_s(H)]$  for the topology induced by $\ch$, which is a consequence of assumptions \ref{H}. Note that $D(A_h)\supset D(A)$ and therefore $A_h$ is a densely defined symmetric operator. With some abuse of notation we denote its closure by $A_h$. Lemma 7.2.15 in \cite{ABG} proves that $A_h$ is self-adjoint and $D(A)$ is a core.

The next step is to add a linear perturbation to $A_h$. Define for each $t\in\R$ the operator $h_s(H)U_t\,h_s(H):=\int_0^t\e^{-\i sH}PK_hP\e^{\i sH}\,\d s$, which is bounded under assumptions \ref{H}. We now show that the limiting operator $B_h:=\slim_{t\to\infty}h_s(H)U_th_s(H)$ is well-defined and bounded as well. Let $\phi,\,\psi$ in $\ch$.
\begin{align}\label{eq:U_Cauchy}\nonumber
&|\braket{\phi}{h_s(H)(U_t-U_r) h_s(H)\psi}|\\ \nonumber
&\qquad=\left|\int_r^t \braket{h_s(H)P\phi}{\e^{-\i sH}K\e^{\i sH}h_s(H)P\psi}\d s\right|\\\nonumber
&\qquad\leq\int_r^t  \big|\braket{F\e^{\i sH}h_s(H)P\phi}{E\e^{\i sH}h_s(H)P\psi}\big|\d s\\\nonumber
&\qquad\leq\left(\int_0^\infty\|F\e^{\i sH}h_s(H)P\phi\|^2\d s\right)^{1/2}\left(\int_r^t\|E\e^{\i sH}h_s(H)P\psi\|^2\d s\right)^{1/2}\\\nonumber
&\qquad\leq  C_F\|h_s(H)P\phi\|_{\dot{H}^s(M)}\left(\int_r^t\|E\e^{\i sH}h_s(H)P\psi\|^2\d s\right)^{1/2}\\
&\qquad\leq  C_F\|\phi\|\left(\int_r^t\|E\e^{\i sH}h_s(H)P\psi\|^2\d s\right)^{1/2}.
\end{align}
Since the integrand of the last step is in $L^1(\R)$ and $\phi$ is arbitrary, we conclude that the sequence converges strongly to the bounded operator $B_h$.

Finally, define $\tilde{A}:=A_h+B_h$. Note that $D(\tilde{A})=D(A_h)\supset D(A)$. The following result justifies this construction.

\begin{proposition}\label{prop:Atilde}
Assume \ref{H}. Then $\tilde{A}$ defined as above is self-adjoint. Moreover, $H$ is of class $C^1(\tilde{A})$ and the continuous sesquilinear form $[H,\i \tilde{A}]_\circ$ on $\ch^1\cap D(\tilde{A})$ is identified with an operator in $B(\ch^1,\ch)$ satisfying the commutation relation $P[H,\i \tilde{A}]P=Q_hP$.
\end{proposition}

\begin{proof} Note that $B_h$ is bounded and symmetric by construction, thus the self-adjointness of $A_h$ is guaranteed by the Kato-Relich theorem. The $C^1(\tilde{A})$ property follows from the identity $[H,A_h]=[H,A]_h$ in $B(\ch^1,\ch^{-1})$ of Prop. 7.2.16, in addition to Prop. 6.2.10 in \cite{ABG}.

We now prove the commutation relation. Using the functional calculus define $R_\epsilon=(1+\i\epsilon H)^{-1}$ and the bounded operator $H_\epsilon:=HH_\epsilon=(\i\epsilon)^{-1}(1-R_\epsilon)$. Note that $P[H_\epsilon,K_h]_\circ P$ is a continuous sesquilinear form on $\ch$ satisfying $P[H_\epsilon,K_h]_\circ P=\i\epsilon^{-1}P[R_\epsilon,K_h]_\circ P=R_\epsilon P[H,K_h]_\circ P R_\epsilon$ in form sense. Now we calculate
\begin{align}\label{eq:commut_HU}\nonumber
[H_\epsilon,\i h_s(H)U_th_s(H)]_\circ &=\frac{\i}{\epsilon}\int_0^t\e^{-\i sH}[R_\epsilon,\i PK_hP]_\circ\e^{\i sH}ds\\ \nonumber
&=R_\epsilon\left(\int_0^t\e^{-\i sH}P[H,\i K_h]_\circ P\e^{\i sH}ds\right)R_\epsilon\\
&=R_\epsilon \e^{-\i tH}PK_hP\e^{\i tH}R_\epsilon-R_\epsilon PK_hPR_\epsilon.
\end{align}
Note that for any $\phi,\,\psi\in\ch^1$ one has $(\phi,\e^{-\i tH}PK_hP\e^{\i tH}\psi)\to 0$ for a subsequence $t_k\to\infty$. This follows from the estimate
\begin{align*}
\int_0^\infty|(\phi,\e^{-\i tH}PK_hP\e^{\i tH}\psi)|\d t &=\int_0^\infty|(F\eith h_s(H) P\phi,E\eith h_s(H)P\psi)|\d t\\
&\leq C \int_0^\infty\big(\big\|F\eith h_sP\phi\big\|^2+(\big\|F\eith h_sP\phi\big\|^2\big)\d t\\
&\leq C_F\|\phi\|^2+C_E\|\psi\|^2.
\end{align*}

Then let $\epsilon\to 0$ on the RHS of \eqref{eq:commut_HU} which converges to $-PK_hP$ in the weak form sense on $\mathcal{H}$. Finally, by making $t_k\to\infty$ in $[H_\epsilon,\i h_s(H)U_{t_k}h_s(H)]_\circ$ and then $\epsilon\to 0$ in weak form sense in $\ch^1$, we conclude that the form $[H,\i B_h]_\circ$ with domain $\ch^1$ extends to a bounded sesquilinear form on $\ch$ satisfying $[H,\i B_h]_\circ=-PK_hP$.

Now, as sesquilinear forms in $ D(H)\cap D(\tilde{A})$ we have $[H,\i \tilde{A}]_\circ=[H,\i A_h]_\circ+[H,\i B_h]_\circ$. By the previous discussion we conclude that $[H,\i \tilde{A}]_\circ$ is continuous for the topology induced by $H$ on $\ch^1\cap D(\tilde{A})$ and moreover
\begin{eqnarray*}
P[H,\i \tilde{A}]P &=& P[H,\i A]_hP-PK_hP\\
&=&Ph_s(H)(Q+K)h_s(H)P_hP-PK_hP\\
&=&PQ_hP,
\end{eqnarray*}
which concludes the proof.\qed
\end{proof}

Once we have established a suitable commutation relation for $H$ and $\tilde{A}$ in Proposition \ref{prop:Atilde}, we recall two important commutator identities used in our previous work. The proofs are analogous and use $\tilde{A}$ as the conjugate operator.

\begin{proposition}
Let $H$ be a self-adjoint operator of class $C^1(\tilde{A})$.  Then the restriction of $[\tilde{A},\e^{\i tH}]_\circ$ to $D(\tilde{A})\cap\ch^1$ extends to a continuous form $[\tilde{A},\e^{\i tH}]$ on  $\ch^1$ and, in the strong topology of space of sesquilinear forms on $\ch^1$, we have $$[\e^{\i tH},\tilde{A}]=\int_0^t\e^{\i(t-s)H}[H,\i\tilde{A}]\e^{\i sH}\d s.$$ Moreover, under the conditions of \ref{H} we have $P[\e^{\i tH},\tilde{A}]P=tPQ_hP\e^{\i tH}$ as operators in $B(\ch^1,\ch)$.
\end{proposition}

\begin{proof}
Same as Theorem 3.7 in \cite{GLS}.\qed
\end{proof}

\begin{proposition}\label{prop:resolvent}
Let $H$ be a self-adjoint operator of class $C^1(\tilde{A})$. Then $D(\tilde{A})\cap\ch^1$ is dense in $\ch^1$ and $$[\tilde{A},R(z)]=-R(z)[\tilde{A},H]R(z)\quad\text{for all }z\in\rho(H).$$
\end{proposition}

\begin{proof}
This is Proposition 6.2.10 in \cite{ABG}.\qed
\end{proof}

In the next section, it will be useful to ``commute $A$ through $B$'' in the important case $Q=cH$, for some $c\neq 0$. To give precise meaning to this, we consider the sesquilinear form $C_\circ$ on $D(\tilde{A})$ defined as $C_\circ(\phi,\psi):=P[\tilde{A},B_h]_\circ P(\phi,\psi)=(\tilde{A}P\phi,B_hP\psi)-(B_hP\phi,\tilde{A}P\psi)$. The next result justifies our assertion.

\begin{proposition}\label{prop:AB-BA}
Under conditions \ref{H} and \ref{Ha} in the case $Q=cH$, the sesquilinear form $C_\circ$ defined on $D(\tilde{A})$ as above can be extended to a bounded operator in $\ch$, which we denote by $C$.
\end{proposition}

\begin{proof}
\underline{Step 1}: Commutator identity

As before, define $R_\epsilon=(1+\i\epsilon H)^{-1}$ and the bounded operator $H_\epsilon:=HH_\epsilon$. For $\phi,\psi$ in $D(\tilde{A})$ consider the form
\begin{eqnarray}\label{eq:C_eps}\nonumber
D_\epsilon(\phi,\psi)&:=&\underbrace{\i((\tilde{A}H_\epsilon P\phi,B_hR_\epsilon^* P\psi)-(B_hH_\epsilon P\phi,\tilde{A}PR_\epsilon^*\psi))}_{D_\epsilon^1}\\
& &\underbrace{-\i((\tilde{A}PR_\epsilon\phi,B_hH_\epsilon^* P\psi)-(B_hR_\epsilon P\phi,\tilde{A}H_\epsilon^*P\psi))}_{D_\epsilon^2}.
\end{eqnarray}
We now use Prop. 3.8 in \cite{GLS} to calculate the commutator identities $P[H_\epsilon,\i \tilde{A}]P=cH_hR_\epsilon^2P$ and $[H_\epsilon,\i B_h]=-PR_\epsilon K_hR_\epsilon P$, which will be used to expand the expression above.
\begin{eqnarray*}
D_\epsilon^1&=&(P[\i \tilde{A},H_\epsilon]P\phi,B_hR_\epsilon^*P\psi)+\i(H_\epsilon \tilde{A}P\phi,B_hR_\epsilon^*\psi)-\i(P\phi,H_\epsilon^*B_h\tilde{A}PR_\epsilon^*\psi)\\
&=&-(cH_hR_\epsilon^2P\phi,B_hR_\epsilon^*P\psi)-(\tilde{A}P\phi,H_\epsilon^*\i B_hR_\epsilon^*\psi)+(P\phi,H_\epsilon^*\i B_h \tilde{A}PR_\epsilon^*\psi)\\
&=&-(cH_hR_\epsilon^2P\phi,B_hR_\epsilon^* P\psi)-(\tilde{A}P\phi,[H_\epsilon^*,\i B_h]R_\epsilon^*\psi)\\
& & -(\tilde{A}P\phi,\i B_hH_\epsilon^*R_\epsilon^*\psi)+(P\phi,[H_\epsilon^*,\i B_h]\tilde{A}PR_\epsilon^*\psi)+(P\phi,\i B_hH_\epsilon^*\tilde{A}PR_\epsilon^*\psi)\\
&=& -(cH_hR_\epsilon^2P\phi,B_hR_\epsilon^*P\psi)+(\tilde{A}P\phi,PR_\epsilon^*K_h(R_\epsilon^*)^2P\psi)\\
& & -(P\phi,PR_\epsilon^*K_hR_\epsilon^*P\tilde{A}PR_\epsilon^*\psi)-(\tilde{A}P\phi,\i B_hH_\epsilon^*R_\epsilon^*\psi)\\
& &+(P\phi,B_hP[H_\epsilon^*,\i \tilde{A}]PR_\epsilon^*\psi)+(P\phi,B_hP\i\tilde{A}H_\epsilon^*PR_\epsilon^*\psi)\\
&=& -(cH_hR_\epsilon^2P\phi,B_hR_\epsilon^*P\psi)+(\tilde{A}P\phi,PR_\epsilon^*K_h(R_\epsilon^*)^2P\psi)\\
& & -(PR_\epsilon K_hR_\epsilon P\phi,\tilde{A}PR_\epsilon^*\psi)+(B_hP\phi,cH_h(R_\epsilon^*)^2PR_\epsilon^*\psi)-D_\epsilon^2.
\end{eqnarray*}
Thus, $D_\epsilon(\phi,\psi)=[B_h,cH_hR_\epsilon^2P]_\circ(\phi,\bar{\psi}_\epsilon)+[\tilde{A}P,PR_\epsilon K_hR_\epsilon P]_\circ(\phi,\bar{\psi}_\epsilon)$, where $\bar{\psi}_\epsilon=R_\epsilon^*\psi$. Note that the expression on the RHS converges weakly in form sense as $\epsilon\to 0$.

The first term of the limit can be extended to the bounded operator $-\i cPK_{h^2}P$ (see proof of Prop. \ref{prop:Atilde}). On the other hand, by conditions \ref{H} and \ref{Ha} the second term can be expanded into Kato-smooth operators
\begin{align*}
&[\tilde{A}P,PK_hP]_\circ(\phi,\psi)\\
&\qquad=[A_hP,PK_hP]_\circ(\phi,\psi)+[B_hP,PK_hP]_\circ(\phi,\psi)\\
&\qquad=(F'Ph\phi,E'Ph\psi)+(FPhB_hP\phi,EPh\psi)-(EPh\phi,FPhB_hP\psi),
\end{align*}
and thus the sesquilinear form can be extended to a bounded operator in $\ch$ satisfying $\|[\tilde{A}P,PK_hP]\|\leq C(E,F,E',F')$.

We conclude that $D_\epsilon$ converges to a sesquilinear form that can be extended to a bounded operator in $\ch$.

\underline{Step 2}: Convergence as $t\to\infty$

Fix $\epsilon>0$ and for $\phi,\psi\in D(\tilde{A})$ denote $\phi_\epsilon=R_\epsilon\phi$, $\psi_\epsilon=R_\epsilon\psi$. Define the sesquilinear form $C_t(\phi,\psi):=(\tilde{A}\e^{\i tH}P\phi,B_h\e^{\i tH}P\psi)-(B_h\e^{\i tH}P\phi,\tilde{A}\e^{\i tH}P\psi)$. Note that $\frac{d}{dt}C_t(\phi_\epsilon,\bar{\psi}_\epsilon)=D_\epsilon(\e^{\i tH}\phi,\e^{\i tH}\psi)$. Now we use an estimate similar to \eqref{eq:U_Cauchy} to prove convergence in $t$.
\vspace{0.1cm}
\begin{align*}
&|(C_t-C_r)(\phi_\epsilon,\bar{\psi}_\epsilon)|\\
&\quad=\left|\int_r^t D_\epsilon(\e^{\i\tau H}\phi,\e^{\i\tau H}\psi)\d\tau\right|\\
&\quad\leq \int_r^t\left|\braket{P\phi_\epsilon}{\e^{\i\tau H}cK_{h^2}\e^{\i\tau H}P\psi_\epsilon}\right|\d\tau+\int_r^t\left|\braket{\phi}{[\tilde{A}P,PR_\epsilon K_hR_\epsilon P]\psi}\right|\d\tau\\
&\quad\leq C(F)\|\phi\|\left(\int_r^t\|E\e^{\i\tau H}h_sP\psi_\epsilon\|^2\d s\right)^{\mkern-8mu 1/2}\mkern-18mu+ C(F')\|\phi\|\left(\int_r^t\|E'\e^{\i\tau H}h_sP\psi_\epsilon\|^2\d s\right)^{\mkern-8mu 1/2}.
\end{align*}

By Kato-smoothness of the RHS we conclude that the sequence is Cauchy in $t$ (for fixed $\epsilon$) and moreover, $C_t(\phi_\epsilon,\psi_\epsilon)\to 0$ as $t\to\infty$.

We proceed analogously to prove the result of the theorem.
\begin{eqnarray*}
|C_0(\phi_\epsilon,\psi_\epsilon)| &\leq& \limsup_{t}\left\{|C_t(\phi,\psi)|+\int_0^t \left|D_\epsilon(\e^{\i\tau H}\phi,\e^{\i\tau H}\psi)\right|\d\tau\right\}\\
&\leq& C(E,F,E',F')\|\phi_\epsilon\|\|\psi_\epsilon\|.
\end{eqnarray*}
Since $R_\epsilon D(A)$ is dense in $\ch$, we conclude by letting $\epsilon\to 0$ that the sesquilinear form defined above $C_\circ(\phi,\psi)=(\tilde{A}P\phi,B_hP\psi)-(B_hP\phi,\tilde{A}P\psi)$ restricted to $D(\tilde{A})$ extends to a bounded operator on $\ch$. \qed

\end{proof}

\begin{corollary}
$B_h$ leaves invariant the domain of $\tilde{A}$, that is, $B_hD(\tilde{A})\subset D(\tilde{A})$.
\end{corollary}
\begin{proof}Let $\epsilon>0$ and $\psi\in D(\tilde{A})$, $\phi\in\ch$. Denote $R_\epsilon(\tilde{A})=(1+\i\epsilon \tilde{A})^{-1}$ and $\tilde{A}_\epsilon=\tilde{A}R_\epsilon(\tilde{A})$. Recall also that $PB_hP=B_h$ and $[P,\tilde{A}]$ extends to a bounded operator since $H\in C^1(\tilde{A})$.
\begin{align*}
&|\braket{\phi}{\tilde{A}_\epsilon B_h\psi}|\\
&\quad = |\braket{P\tilde{A}R_\epsilon(\tilde{A})^*\phi}{B_hP\psi}|\\
&\quad = |\braket{[P,\tilde{A}](1-\i\epsilon \tilde{A})^{-1}\phi}{B_hP\psi}+\braket{\tilde{A}P(1-\i\epsilon \tilde{A})^{-1}\phi}{B_hP\psi}|\\
&\quad\leq C\|\phi\|\|\psi\|+|\braket{(1-\i\epsilon \tilde{A})^{-1}\phi}{P[\tilde{A},B_h]P\psi}|+|\braket{B_hP(1-\i\epsilon \tilde{A})^{-1}\phi}{P\tilde{A}P\psi}|\\
&\quad\leq C\|\phi\|(\|\psi\|+\|\tilde{A}\psi\|),
\end{align*}
where the constant $C$ is independent of $\epsilon$. Thus, $\|\tilde{A}_\epsilon B_h\psi\|\leq C(\|\psi\|+\|\tilde{A}\psi\|)$ and we conclude by Fatou's lemma.\qed
\end{proof}

\section{Decay estimates}

\begin{definition}
For $u\in\mathcal{H}$, define the function $\psi_u(t):=\langle u,e^{\i tH}u\rangle$, $t\in\mathbb{R}$ and the set $$\ce=\big\{u\in \mathcal{H}: \psi_u\in L^2(\mathbb{R}) \big\}.$$
For $u\in\ce$ denote $[u]_H=\|\psi_u\|_{L^2_t}^{1/2}$.
\end{definition}

In \cite{ABG} it was shown that $\ce$ is a dense linear subspace of the absolutely continuity subspace of $H$ and $[\cdot]_H$ is a complete norm on it. In this work we will consider the following additional assumption on the space $\ce$, which will be relevant for Propositions \ref{prop:t_decay_1/2} and \ref{prop:t_decay_1/2_s=0} where no asssumption of positivity of $Q$ is made.

\begin{itemize}
\labitemi{(Hb)}{Hb} For any $u\in\ch$ one has $(A_h+\i)^{-1}u\in\ce$.
\end{itemize}

We now proceed to prove the main results of this work.

\begin{proposition}\label{prop:t_decay_one}
Assume \ref{H} with $Q\geq 0$. Then for $u\in\ch^s$ such that $Pu\in D(A)\cap D(Q^{1/2})$ we have the estimate $|\psi_{Q^{1/2}Pu}(t)|\leq C_u\jap{t}^{-1}$.
\end{proposition}

\begin{proof} Assume first that $u\in\ch^1$ and define $v=\jap{H}^{s/2}u$.
\begin{eqnarray}\label{eq:Atildev}\nonumber
t\psi_{Q^{1/2}Pu}(t) &=&\braket{Q^{1/2}Pu}{t\e^{\i tH}Q^{1/2}Pu}\\\nonumber
&=&\braket{v}{tPQ_h\e^{\i tH}Pv}\\\nonumber
&=&\braket{v}{P[\e^{\i tH},\tilde{A}]Pv}\\
&=&\braket{\e^{-\i tH}Pv}{\tilde{A}Pv}-\braket{\tilde{A}Pv}{\e^{\i tH}Pv}.
\end{eqnarray}
We now expand the first term of the last expression, the second one is analogous.
\begin{eqnarray*}
|\braket{\e^{-\i tH}Pv}{\tilde{A}Pv}| &=&|\braket{\e^{-\i tH}Pv}{A_hPv}+\braket{\e^{-\i tH}Pv}{B_hPv}|\\
&\leq& \|u\|\|PAPu\|+C\|\jap{H}^{s/2}Pu\|.
\end{eqnarray*}

Hence $|t\psi_{Q^{1/2}Pu}(t)|\leq 2(\|u\|\|PAPu\|+C\|\jap{H}^{s/2}Pu\|)$.

For general $u$ we define $u_\epsilon:=R_\epsilon u\in\ch^1$. Note that $Pu_\epsilon\in D(Q^{1/2})$ since $Q$ commutes with $H$ and thus from \eqref{eq:Atildev} we obtain the estimate
\begin{equation}\label{eq:decay_1_eps}
|t\psi_{Q^{1/2}Pu_\epsilon}(t)|\leq 2\|v_\epsilon\|\|P\tilde{A}Pv_\epsilon\|.
\end{equation}
Since $H\in C^1(\tilde{A})$, from Proposition 12 in \cite{GLS} we obtain that $PR_\epsilon v\in D(\tilde{A})$. Moreover, Proposition \ref{prop:resolvent} implies that $[\tilde{A},R_\epsilon]=\epsilon R_\epsilon[H,\i \tilde{A}]R_\epsilon$, hence $P[\tilde{A},R_\epsilon]P=\epsilon PR_\epsilon QR_\epsilon P$ as operators in $\ch$. Now we use that $Q$ is $H$-bounded with relative norm $a$, which yields
\begin{eqnarray*}
\|P\tilde{A}Pv_\epsilon\| &=& \|P[\tilde{A},R_\epsilon]Pv\|+\|PR_\epsilon \tilde{A}Pv\|\\
&\leq& \epsilon\left(a\|HR_\epsilon Pv\|+b\|R_\epsilon Pv\|\right)+\|\tilde{A}Pv\|\\
&\leq& \left(a+\epsilon b\right)\|Pv\|+\|PAPu\|+C\|Pv\|\\
&\leq& C\|\jap{H}^{s/2}Pu\|+\|PAPu\|.
\end{eqnarray*}
Finally, let $\epsilon\to 0$ and use Fatou's lemma on the lhs side of \eqref{eq:decay_1_eps} to conclude
\begin{eqnarray*}
|t\psi_{Q^{1/2}Pu}(t)|&\leq& 2\|\jap{H}^{s/2}u\|\big(C\|\jap{H}^{s/2}Pu\|+\|PAPu\|\big).
\end{eqnarray*}\qed
\end{proof}

\begin{proposition}\label{prop:t_decay_1/2}
Assume \ref{H}, \ref{Ha} and \ref{Hb} in the special case $Q=cH$ and $s=1/2$. Then for $u\in D(A)$ such that $Pu$ and $PAPu$ are in $\ce$, one has $|\psi_{Pu}(t)|\leq C_u\jap{t}^{-1/2}$.
\end{proposition}

\begin{proof} Define $\chi=\chi_{[-M,1]}$ the characteristic function of the interval in $\R$. We decompose $u:=u_1+u_2$, where $u_1=\chi(H)u$ and $u_2=(1-\chi(H))u$. Note that $\psi_{Pu}=\psi_{Pu_1}+\psi_{Pu_2}$. We will show that $\psi_{u_1}=O(t^{-1/2})$ and $\psi_{u_2}=O(t^{-1})$.

By Corollary 8.2 in \cite{GLS} it suffices to prove that $\delta\psi_{Pu_1}(t):=t\psi_{Pu_1}'(t)$ is in $L^2(\R)$. Define $v=\jap{H}^{s/2}u_1$.
\begin{eqnarray}\label{eq:Atildev2}\nonumber
\i ct\psi_{Pu_1}'(t) &=& \braket{Pu_1}{tcH\e^{\i tH}Pu_1}\\\nonumber
&=& \braket{Pv}{t\jap{H}^{-s/2}cH\jap{H}^{-s/2}\e^{\i tH}Pv}\\\nonumber
&=& \braket{Pv}{tH_he^{\i tH}Pv}\\\nonumber
&=& \braket{v}{P[\e^{\i tH},\tilde{A}]Pv}\\\nonumber
&=& \braket{\e^{-\i tH}v}{P\tilde{A}Pv}-\braket{P\tilde{A}Pv}{\e^{\i tH}v}.
\end{eqnarray}
We now expand the first term of the last expression, the second one is analogous.
\begin{eqnarray*}
\braket{\e^{-\i tH}v}{P\tilde{A}Pv} &=& \braket{\e^{-\i tH}v}{PA_hPv}+\braket{\e^{-\i tH}v}{PB_hPv}\\
&=&\braket{\e^{-\i tH}u_1}{PAPu_1}+\braket{\e^{-\i tH}Pv}{(A_h+\i)^{-1}(A_h+\i)B_hPv}.
\end{eqnarray*}
Note that $(A_h+\i)B_hPv=[A_h,B_h]Pv+B_h(A_h+\i)Pv$, which is in $\ch$ by Proposition \ref{prop:AB-BA} and the fact that $v\in D(A_h)$.

Thus $$c\|\delta\psi_{Pu_1}\|\leq 2[Pu_1]_H\big([PAPu_1]_H+[(A+\i)^{-1}(CP\jap{H}^{s/2}u_1+B_h\jap{H}^{-s/2}APu_1)]_H\big),$$ and we conclude $|\psi_{u_1}(t)|\leq C_{u_1}\jap{t}^{-1/2}$.

To estimate the decay of $u_2$ we now consider $v\in\ch$ such that $u_2=|H|^{1/2}\jap{H}^{-1/2}v$. Then
\begin{eqnarray}\label{eq:Atildev3}\nonumber
ct\psi_{Pu_2}(t)&=&\braket{Pu_2}{ct\e^{\i tH}Pu_2}\\\nonumber
&=&\braket{Pv}{ct\e^{\i tH}H_h\,\textrm{sgn}(H)Pv}\\
&=&\braket{v}{P[\e^{\i tH},\tilde{A}]\,\textrm{sgn}(H)Pv}.
\end{eqnarray}
Note that $g(H):=|H|^{-1/2}\jap{H}^{1/2}(1-\chi(H))$ is a bounded smooth function, so $\|v\|\leq\|u\|$ and $P\tilde{A}g(H)P=g'(H)H_hP+Pg(H)\tilde{A}P$. Hence the first term of the commutator in \eqref{eq:Atildev3} is bounded  $|\braket{\e^{-\i tH}v}{P\tilde{A}Pv}|\leq\|u\|\big(\|u\|+\|\tilde{A}Pu\|\big)$ and we conclude that $|\psi_{Pu_2}(t)|\leq C_u\jap{t}^{-1}$ as desired.
\qed
\end{proof}

\begin{remark}
In case assumptions \ref{H}, \ref{Ha} and \ref{Hb} are met with $s=0$, the construction of the conjugate operator is simpler because there is no need to introduce the cut-off $h_s$ (and then $\tilde{A}=A+B$). Hence Propositions \ref{prop:t_decay_one} and \ref{prop:t_decay_1/2} remain valid. We state them here for completeness, the proofs are analogous.
\end{remark}

\begin{proposition}
Assume \ref{H} with $Q\geq 0$ and $s=0$. Then for $u\in\ch$ such that $Pu\in D(A)\cap D(Q^{1/2})$ we have the estimate $|\psi_{Q^{1/2}Pu}(t)|\leq C_u\jap{t}^{-1}$.
\end{proposition}

\begin{proposition}\label{prop:t_decay_1/2_s=0}
Assume \ref{H}, \ref{Ha} and \ref{Hb} in the special case $Q=cH$ and $s=0$. Then for $u\in D(A)$ such that $Pu$ and $PAPu$ are in $\ce$, one has $|\psi_{Pu}(t)|\leq C_u\jap{t}^{-1/2}$.
\end{proposition}

We now study the particular case  $H=H_0+V(x)$ on $\ch=L^2(\R^n)$, where $H_0$ is a self-adjoint operator and $V(x)$ is smooth real-valued function. Let $P$ be a projection that commutes with $H$. Let us consider the following assumptions.

\begin{itemize}
\labitemi{(H1)}{H1}
\begin{itemize}
\item[(i)] There is a self-adjoint first-order operator $A$ so that $H_0$ is of class $C^1(A)$ and $[H_0,\i A]=cH_0$ for some $c\neq 0$
\item[(ii)] The functions $V(x)$ and $w(x):=[V(x),\i A]$ are bounded
\end{itemize}
\end{itemize}

Conditions \ref{H1} ensure that $H$ is of class $C^1(A)$ and it follows that the sesquilinear form $[H,\i A]_\circ=cH-cV(x)+w(x)$ on $D(H)\cap D(A)$ can be identified with an operator $[H,\i A]$ in $B(\ch^1,\ch)$. In order to satisfy assumptions \ref{H} with $Q=cH$ and $K=-cV(x)+w(x)$ it remains to show the Kato-smoothness condition. We can derive this from decay estimates as follows.

Let $\sigma>0$ and $P_{\text{ac}}(H)$ be the projection onto the space of absolute continuity of $H$. Consider the local decay estimate
\begin{equation}\label{eq:LC}
\int_{\R}\|\jap{x}^{-\sigma/2}\e^{-\i tH}P_{\text{ac}}u\|^2\d t\leq C\|u\|^2,
\end{equation}
for all $u\in\ch$ and some $C>0$. Then from Proposition \ref{prop:t_decay_1/2} we obtain the following result.

\begin{proposition}
Let $H=H_0+V(x)$ and $A$ as in \ref{H1}, with $V$ such that $$\sup_{x\in\R^n}\left(\jap{x}^\sigma|V(x)|+\jap{x}^{\sigma+1}|\nabla V(x)|\right)<\infty.$$ Assume also the local decay estimate \eqref{eq:LC}.  Then for $u\in D(A)$ such that $P_\text{ac}AP_\text{ac}u$ is in $\ce$, then $|\psi_{P_{\text{ac}}u}(t)|\leq C_u\jap{t}^{-1/2}$.
\end{proposition}

\section{Higher-order decay estimates}

We now improve our results by iteration of the previous method. In order to obtain higher-order decay estimates, the main difficulty lies in extending Proposition \ref{prop:AB-BA} for higher powers of $A_h$ and $B_h$. This task is extremely laborious with the current methods so it will not be pursued here. Proposition \ref{prop:AB-BA} allows to construct the operator $\tilde{A}^2=(A_h+B_h)^2$ on $D(\tilde{A}^2)=D(A_h^2)\supset D(A^2)$. This will be enough to increase the time decay rate by a power of one as shown in the propositions below.

In this section, we will restrict ourselves to the case $s=0$ and $P=\textrm{Id}$ in \ref{H}, \ref{Ha} and \ref{Hb}. So here $\tilde{A}=A+B$ defined on $D(\tilde{A})=D(A)$, $B$ is bounded on $\ch$ and the commutation relation reads $[H,\tilde{A}]=Q$.

We now recall the necessary formalism of higher-order regularity of operators (see \cite{ABG} and \cite{GLS} for further discussion). Let $A$ be a self-adjoint operator on a Hilbert space $\ch$ and
$k\in\N$.  We say that a bounded operator \emph{$S$ is of class $C^k(A)$}, and we write
$S\in C^k(A)$, if the map
$ \R\ni t\mapsto \e^{-\i t A}S\e^{\i tA}S\in B(\ch)$ is of class $C^k$
in the strong operator topology. It is clear that $S\in C^{k+1}(A)$ if
and only if $S\in C^{1}(A)$ and $[S,A]\in C^{k}(A)$. Clearly $C^{k}(A)$ is a
$*$-subalgebra of $ B(\ch)$ and if $S\in B(\ch)$ is bijective and
$S\in C^{k}(A)$ then $S^{-1}\in C^{k}(A)$,

For any $S\in B(\ch)$ let $\tilde{\ca}(S)=[S,\i A]$ considered as a
sesquilinear form on $D(A)$. We may iterate this and define a
sesquilinear form on $D(\tilde{A}^k)$ by:
\[
S^{(k)}\equiv\tilde{\ca}^k(S)=\i^k\sum_{i+j=k} \frac{k!}{i!j!} (-\tilde{A})^i S \tilde{A}^j.
\]
Then $S\in C^k(A)$ if and only if this form is continuous for the
topology induced by $\ch$ on $D(A^k)$. We keep the notation $\tilde{\ca}^k(S)$
or $S^{(k)}$ for the bounded operator associated to its continuous
extension to $\ch$.

Now let $H$ be a self-adjoint operator on $\ch$ and
$ R(z)= (H-z)^{-1}$ for $z$ in the resolvent set $\rho(H)$ of $H$.  We
say that \emph{$H$ is of class $C^k(A)$} if $R(z_{0})\in C^k(A)$ for
some $z_{0}\in\rho(H)$; then we shall have $R(z)\in C^k(A)$ for all
$ z\in \rho(H)$ and more generally $\varphi(H)\in C^k(A)$ for a large
class of functions $\varphi$ (e.g. rational and bounded on the
spectrum of $H$).		

We shall say that a densely defined operator $S$ on $\ch$ is \emph{boundedly invertible} if $S$ is injective,
its range is dense, and its inverse extends to a continuous operator on $\ch$. If $S$ is symmetric
this means that $S$ is essentially self-adjoint and 0 is in the resolvent set of its closure.

\begin{proposition}\label{prop:t=-2_decay}
Let $H$ be of class $C^2(A)$. Assume \ref{H} and \ref{Ha} with $s=0$ and $Q=Q(H)$ such that $Q'$ is a bounded function. Then for $u\in\ch$ such that $u\in D(A^2)\cap D(Q)$ we have the estimate $|\psi_{Qu}(t)|\leq C_u\jap{t}^{-2}$.
\end{proposition}

\begin{proof}
Note that as sesquilinear forms in $D(Q^2)\cap D(\tilde{A})$ one has the identity $$t^2Q^2\e^{\i tH}=\tilde{\ca}^2(\e^{\i tH})-Q'(H)\tilde{\ca}(\e^{\i tH}).$$
Note also that $\tilde{\ca}^2(H)=Q'(H)Q(H)$, hence $H\in C^2(\tilde{A})$ by Propositions 7.2.16 and 6.2.10 in \cite{ABG}.

Assume first that $u\in\ch^1$, then
\begin{eqnarray*}
|t^2\psi_{Qu}(t)|&=&|\braket{u}{t^2Q^2\e^{\i tH}u}|\\
&\leq&|\braket{u}{\tilde{\ca}^2(\e^{\i tH})u}|+|\braket{u}{Q'(H)\tilde{\ca}(\e^{\i tH})u}|\\
&\leq& C(\|u\|\|\tilde{A}^2u\|+\|\tilde{A}u\|^2)+\|u\|\|\tilde{A}Q'(H)u\|\\
&\leq&C\|u\|(\|\tilde{A}u\|+\|\tilde{A}^2u\|)+\|\tilde{A}u\|^2.
\end{eqnarray*}
By Proposition \ref{prop:AB-BA}, $(A+B)^2u$ is well defined and $\|\tilde{A}^2u\|\leq\|A^2u\|+C\|Au\|+\|u\|$, thus $|\psi_{Qu}(t)|\leq C_u\jap{t}^{-2}$, with $C_u=C\|u\|\big(\|u\|+\|Au\|+\|A^2u\|\big)+\|Au\|^2$.

For general $u\in D(A^2)\cap D(Q)$ we use $u_\epsilon=R_\epsilon u\in\ch^1$. Note that $u_\epsilon\in D(\tilde{A}^2)$ because $R_\epsilon\in C^2(\tilde{A})$. Proceeding like in the proof of Proposition \ref{prop:t_decay_one} and letting $\epsilon\to 0$ we obtain the desired result.\qed
\end{proof}

\begin{remark}
Note that if the operator $Q(H)$ is assumed boundedly invertible, the estimate of Proposition \ref{prop:t=-2_decay} holds for $\psi_u$(t), with $u\in D(A^2)\cap D(Q)$.
\end{remark}

\begin{proposition}\label{prop:t_decay_3/2_s=0}
Let $H$ be of class $C^2(A)$. Assume \ref{H}, \ref{Ha} and \ref{Hb} in the special case $Q=cH$ and $s=0$. Let  $u\in\ch$ be of the form $u=|H|^{1/2}v$, for some $v\in D(A^2)$ such that $A^jv\in\ce$, $j=0,1,2$ and $ABv\in\ce$. Then $|\psi_{u}(t)|\leq C_u\jap{t}^{-3/2}$.
\end{proposition}

\begin{proof}
 Define $\chi=\chi_{[-M,1]}$ the characteristic function of the interval in $\R$. We decompose $u:=u_1+u_2$, where $u_1=\chi(H)u$ and $u_2=(1-\chi(H))u$. Note that $\psi_{Pu}=\psi_{Pu_1}+\psi_{Pu_2}$. We will show that $\psi_{u_1}=O(t^{-3/2})$ and $\psi_{u_2}=O(t^{-2})$. To study $\psi_{u_1}$ we rely on Corollary 8.3 in \cite{GLS}, that is, we need to prove that both $t\psi_{u_1}(t)$ and $t^2\psi_{u_1}(t)$ are functions in $L^2(\R)$. We only show the latter, the former is an analogous calculation. Set $v_1=\chi(H) v$.
\begin{eqnarray*}
|\i c^2t^2\psi_{u_1}'(t)| &=& |\braket{u_1}{c^2t^2H\e^{\i tH}u_1}|\\
&=&|\braket{v_1}{c^2t^2H^2\e^{\i tH}\sign(H)v_1}|\\
&\leq&|\braket{v_1}{\tilde{\ca}^2(\e^{\i tH})v_1}|+|c\braket{v_1}{\tilde{\ca}(\e^{\i tH})v_1}|.
\end{eqnarray*}
We now expand the first term of the above expression, the second one is analogous. Since $\chi(H)$ is a bounded function and $H\in C^2(\tilde{A})$, $\chi(H$) it can be commuted through $\tilde{A}$ and so we can replace $v_1$ with $v$ (up to a bounded function of $H$).
\begin{eqnarray*}
\braket{v}{\tilde{\ca}^2(\e^{\i tH})v}=\braket{\e^{-\i tH}v}{\tilde{A}^2v}-2\braket{\tilde{A}v}{\e^{\i tH}\tilde{A}v}+\braket{\tilde{A}^2v}{\e^{\i tH}v}
\end{eqnarray*}
Note that $\tilde{A}v=Av+(A+\i)^{-1}(A+\i)Bv\in\ce$ by assumption \ref{Hb}. Similarly, $\tilde{A}^2v\in\ce$ because
\begin{eqnarray*}
\tilde{A}^2v&=&A^2v+ABv+BAv+B^2v\\
&=&A^2v+ABv+(A+\i)^{-1}(A+\i)BAv+(A+\i)^{-1}(A+\i)B^2v.
\end{eqnarray*}
This yields that $t^2\psi_{u_1}'(t)\in L^2(\R)$ as desired. Now for $u_2$, we consider $v_2\in\ch$ such that $u_2=cH\jap{H}^{-1}v_2$. Then
\begin{eqnarray*}
|t^2\psi_{u_2}(t)|&=&|\braket{u_2}{t^2\e^{\i tH}u_2}|\\
&=&|\braket{v_2}{c^2t^2H^2\e^{\i tH}\jap{H}^{-2}v_2}|\\
&\leq& C(\|v_2\|\|\tilde{A}^2v_2\|+\|\tilde{A}v_2\|^2)\\
&\leq& C\|u_2\|(\|u_2\|+\|Au_2\|+\|A^2u_2\|)+\|Au_2\|^2,
\end{eqnarray*}
which concludes the proof.\qed
\end{proof}

\section{Applications}

\subsection{Potential of critical decay}
Here we consider the equation $\i\partial_tu+\Delta u-V(x)u=0$ in $\R^n$ $(n\geq 3)$, with initial condition $u(0,x)=f(x)$. In \cite{BPSS}, a resolvent estimate is used to obtain weighted $L^2$ estimates for time-independent potentials $V(x)\in C^1(\R^n\setminus\{0\})$ satisfying the following assumptions.
\begin{enumerate}
\item[(A1)]$\sup_{x\in\R^n}|x|^2|V|<\infty$
\item[(A2)] The operator $\Delta\!\!\!\!/+|x|^2V+\lambda^2$ is positive on every sphere, i.e., there is a $\delta>0$ such that for every $r>0$, $$\int_{|x|=r}|\nabla\!\!\!\!/u(x)|^2+(\lambda^2+|x|^2V(x))|u(x)|^2\d\sigma(x)\geq\delta^2\int_{|x|=r}|u(x)|^2\d\sigma(x)$$
\item[(A3)] The operator $\Delta\!\!\!\!/+|x|^2\tilde{V}+\lambda^2$ is positive on every sphere, i.e., (A2) holds with $\tilde{V}$ in place of $V$.
\end{enumerate}
Here $\Delta\!\!\!\!/$ represents the spherical Laplacian, $\tilde{V}:=\partial_r(rV(x))$ and $\lambda=(n-2)/2$.

In Section 3 of that paper the following Morawetz estimate is obtained
\begin{equation}\label{eq:morawetz_sq_potential}
\left\| |x|^{-1}\e^{-\i tH}f\right\|_{L_t^2L_x^2}\leq C\|f\|_{L^2}.
\end{equation}

We will use this result to obtain pointwise estimates using the generator of dilations $A=-\tfrac{i}{2}(x\cdot\nabla+\nabla\cdot x)$ as the conjugate operator. In order to verify the $C^1(A)$ regularity for $H$ and the other conditions of \ref{H} we will consider the following additional assumptions.

\begin{enumerate}
\item[(A4)] $V$ is nonnegative and locally integrable in $\R^n$
\item[(A5)] $\sup_{x\in\R^n}|x|^3|\nabla V|<\infty$ and $\|(x\cdot\nabla V)\jap{V}^{-1}\|_{L^\infty}<\infty$
\end{enumerate}

In \cite{Ka95} Theorem 4.6a the domain of the Friedrich extension of $H=-\Delta+V(x)$ is characterized as the set of all i) $u\in\ch$ such that $\nabla u$ belongs to $(L^2(\R^n))^n$, ii)$\int V(x)|u|^2\d x<\infty$, and iii) $\Delta u$ exists and $-\Delta u+V(x)u$ belongs to $\ch$. It is easy to check that the dilation group leaves $D(H)$ invariant, in fact ii) and iii) are preserved under condition (A5). In this scenario, the $C^1(A)$ regularity follows from the commutation relation $[H,\i A]=-2\Delta-x\cdot\nabla V=2H-(2V+x\cdot\nabla V)$, which clearly satisfies $\|[Hu,Au]_\circ\|\leq C(\|Hu\|+\|u\|)$ by (A4) and (A5). Note that here $Q=2H$ and $K=2V+x\cdot\nabla V$.

Now write $K=\sign(K)|K|^{1/2}|K|^{1/2}$. Assumptions (A1) and (A4) yield the bound $\| |K|^{1/2}\e^{\i tH}f\|_{L^2}\leq C\| |x|^{-1}\e^{\i tH}f\|_{L^2}$, and then estimate \eqref{eq:morawetz_sq_potential} implies that $K$ is the product of two Kato-smooth operators. Therefore, the assumptions of \ref{H} hold with $s=0$ and the estimate of Proposition \ref{prop:t_decay_1/2_s=0} reads as follows.

\begin{proposition}
Let $H$ and $A$ be as above and $P$ a smooth projection commuting with $H$. Assume conditions (A1)-(A5) and \ref{Hb}. Then for $u\in D(A)$ such that $Pu$ and $PAPu$ are in $\ce$, one has the estimate $|\psi_{Pu}(t)|\leq C_u\jap{t}^{-1/2}$.
\end{proposition}

\begin{remark}
We can use the results of \cite{BPSS} to characterize the space $\ce$. For instance, the endpoint Strichartz estimate $\|\e^{-\i tH}u\|_{L^2_t(L^6_x)}\leq C\|u\|_{L^2}$ for $n=3$ yields that for $u\in L^{6/5}(\R^3)$, we have $\|\psi_u\|_{L^2_t}\leq\|u\|_{L^{6/5}}\|\e^{-\i tH}u\|_{L^2_t(L^6_x)}\leq C\|u\|_{L^{6/5}}\|u\|_{L^2}$. Or we can use the estimate \eqref{eq:morawetz_sq_potential} in dimension $n$, to show that if $xu\in L^2(\R^n)$, one has $\|\psi_u\|_{L^2_t}\leq \|xu\|_{L^2}\|u\|_{L^2}$.
\end{remark}

\subsection{Laplacian on manifold} Let $(M,g)=(\R^3,g)$ be a compact perturbation of $\R^3$, i.e. $M$ is $\R^3$ endowed with a smooth metric $g$ which equals the Euclidean metric outside of a ball $B(0,R_0)=\{x\in\R^3:|x|\leq R_0\}$ for some fixed $R_0$. In \cite{RT} global-in-time decay estimates were obtained for solutions to the Schr\"odinger equation $iu_t=Hu$, where $H=-\frac{1}{2}\Delta_M$ is the Laplace-Beltrami operator on $M$. It has been shown that $H$ defined on $\cs:=C^\infty(M)$ (smooth functions of compact support) is essentially self-adjoint and its domain is the Sobolev space $H^2(M)$.
The main result of their paper is the following.

\begin{theorem*}
Let $M$ be a smooth compact perturbation of $\R^3$ which is nontrapping and smoothly diffeomorphic to $\R^3$. Then for any Schwartz solution $u(x,t)$ and any $\sigma>0$ we have
\begin{equation}\label{eq:RodTao}
\int_{\R}\|\jap{x}^{-1/2-\sigma}\nabla\e^{-\i tH}u_0\|^2_{L^2(M)}+\|\jap{x}^{-3/2-\sigma}\e^{-\i tH}u_0\|^2_{L^2(M)}\,dt\leq C_{\sigma,M}\|u_0\|^2_{\dot{H}^{1/2}(M)}.
\end{equation}
\end{theorem*}

We choose the conjugate operator $A=-\i/2(x\cdot\nabla+\nabla\cdot x)$ defined on $\cs$. Note that the dilation group leaves $H^2(M)$ invariant. Assuming that metric is smooth and bounded, the $C^1(A)$ condition follows from the commutation relation $[H,\i A]=2H+K$, where $K$ is a second order operator supported in $B(0,R_0)$. More explicitly, a straightforward calculation on $\cs$ shows that $K$ is an an operator of the form $K=k_1(x)H+\i k_2(x)\nabla+\i k_3(x)$, where the $k_i'$s are smooth and bounded real functions supported in $B(0,R_0)$. Observe also that $\jap{H}^{-1/2}K\jap{H}^{-1/2}$ is bounded. Taking $E=F=\sqrt{|K|}$, the Kato-smoothness is a direct consequence of the estimate \eqref{eq:RodTao} and therefore the assumptions of \ref{H} are met with $s=1/2$. Condition \ref{Ha} holds since $[A,K]$ is a second order differential operator as well. Proposition \ref{prop:t_decay_1/2} yields the following result.

\begin{proposition}
Let $H$ and $A$ be above and $P$ a smooth projection commuting with $H$. Assume that condition \ref{Hb} holds for the space $\ce$. Then for $u\in D(A)$ such that $Pu$ and $PAPu$ are in $\ce$, one has $|\psi_{Pu}(t)|\leq C_u\jap{t}^{-1/2}$.
\end{proposition}

\end{document}